\documentclass[12pt]{amsart}
\usepackage{amssymb}

\newtheorem{thm}{Theorem}[section]
\newtheorem{prp}{Proposition}[section]
\newtheorem{lem}{Lemma}[section]
\newtheorem{cor}{Corollary}[section]
\begin{document}
\author{G\'erard Endimioni}
\address{C.M.I-Universit\'{e} de Provence\\
39, rue F. Joliot-Curie, F-13453 Marseille Cedex 13}
\email{endimion@gyptis.univ-mrs.fr}
\title[Polynomial automorphisms]
{On the polynomial automorphisms of a group}
%
\begin{abstract} Let ${\rm A}(G)$ denote the automorphism group of a 
group $G$. A polynomial automorphism of $G$ is an automorphism of the form
$x\mapsto (v_{1}^{-1}x^{\epsilon_{1}}v_{1})\ldots
(v_{m}^{-1}x^{\epsilon_{m}}v_{m})$. We prove that if $G$ is nilpotent
(resp. metabelian), then so is the subgroup of ${\rm A}(G)$
generated by all polynomial automorphisms. 
\end{abstract}
\maketitle
%
\section{Introduction and main results}
Let $G$ be a group. We shall write ${\rm A}(G)$ for the 
automorphism group of $G$. According to Schweigert \cite{SC}, 
we say that an element $f\in {\rm A}(G)$ is a {\em polynomial automorphism} of $G$
if there exist integers $\epsilon_{1},\ldots,\epsilon_{m}\in{\mathbb Z}$ 
and elements $u_{0},\ldots,u_{m}\in G$ such that
$$f(x)=u_{0}x^{\epsilon_{1}}u_{1}\ldots 
u_{m-1}x^{\epsilon_{m}}u_{m}$$ 
for all $x\in G$. Since $f(1)=1$, it 
is easy to see that $f(x)$ can be expressed as a `product' of inner 
automorphisms, that is
$$f(x)=(v_{1}^{-1}x^{\epsilon_{1}}v_{1})\ldots
(v_{m}^{-1}x^{\epsilon_{m}}v_{m}).$$
We shall write ${\rm P}_{0}(G)$ for the set of polynomial automorphisms of $G$.
Actually, Schweigert defines a polynomial automorphism in the context 
of finite groups. In particular, in this context, the set
${\rm P}_{0}(G)$ is clearly a subgroup of ${\rm A}(G)$. On the other hand, this 
is not necessarily the case when $G$ is 
infinite. For instance, in the additive group of rational numbers, 
the set of polynomial automorphisms forms a monoid with respect to 
the operation of functional composition, which is isomorphic to
the multiplicative monoid ${\mathbb Z}\setminus\{ 0\}$. 

In this paper,
we shall consider the subgroup 
${\rm P}(G)=\langle {\rm P}_{0}(G)\rangle$ of ${\rm A}(G)$, generated by all  
polynomial automorphisms of $G$. 
Hence ${\rm P}_{0}(G)={\rm P}(G)$ when $G$ is finite, but for example
${\rm P}(G)$ is distinct from ${\rm P}_{0}(G)$ when $G$ is
the additive group of rational numbers (note that ${\rm P}(G)={\rm A}(G)$
in this last case). 

It is easy to verify that ${\rm P}_{0}(G)$ is a normal subset of ${\rm A}(G)$.
Thus ${\rm P}(G)$ is a normal subgroup of ${\rm A}(G)$; in addition, 
we have 
$${\rm I}(G)\unlhd {\rm P}(G)\unlhd {\rm A}(G),$$ 
where ${\rm I}(G)$ is the group of inner automorphisms of $G$.
Also ${\rm P}(G)$ contains the group of invertible elements of the 
monoid ${\rm P}_{0}(G)$. It is worth noting that there exist finite groups 
$G$ such that the quotient ${\rm P}(G)/{\rm I}(G)$ is not soluble \cite{KO}.

If $G$ is abelian, each polynomial automorphism is of the form
$x\mapsto x^{\epsilon}$, and so ${\rm P}(G)$ is abelian.
When $G$ is a finite nilpotent group of class $k\geq 2$, it is proved 
in \cite{CR} that ${\rm P}(G)$ is nilpotent of class $k-1$
(see also \cite[Satz 3.5]{SC}).
We show here that this result remains true when $G$ is infinite.
\begin{thm} Let $G$ be a nilpotent group of class $k\geq 2$.
Then ${\rm P}(G)$ is nilpotent of class $k-1$.
\end{thm}
Notice that conversely, if ${\rm P}(G)$ is nilpotent, then so is $G$
since ${\rm P}(G)$ contains the group of inner automorphisms.

When $G$ is metabelian, it seems that nothing is known about 
${\rm P}(G)$, even in the context of finite groups. In this paper, we 
shall prove the following.
\begin{thm}  Let $G$ be a metabelian group.
Then the group  ${\rm P}(G)$ is itself metabelian.
\end{thm}
In Section 3, we shall interpret a result of C. K. Gupta as a very 
particular case of this theorem (see Corollary 3.1 below).
%
\section{Proofs}
As usual, in a group $G$, the commutator of two elements $x,y$ is 
defined by $[x,y]=x^{-1}y^{-1}xy$. Instead of $[[x,y],z]$, we shall 
write $[x,y,z]$. We denote by $[G,G]$ the derived subgroup of $G$.

\begin{lem} Let $f,g$ be two
functions over a group $G$, respectively defined by the relations
\begin{eqnarray*}
	f(x) & = & (v_{1}^{-1}x^{\epsilon_{1}}v_{1})\ldots
     (v_{m}^{-1}x^{\epsilon_{m}}v_{m}),  \\
	g(x) & = & (w_{1}^{-1}x^{\eta_{1}}w_{1})\ldots
(w_{n}^{-1}x^{\eta_{n}}w_{n})
\end{eqnarray*}
(we do not suppose that $f$ and $g$ are automorphisms).
Let $t$ be an element of $G$ such that any two conjugates of $t$ 
commute. Then we have the relation
$$f(g(t))=\prod_{i=1}^{m}\prod_{j=1}^{n}t^{\epsilon_{i}\eta_{j}}
[t^{\epsilon_{i}\eta_{j}},v_{i}]
[t^{\epsilon_{i}\eta_{j}},w_{j}]
[t^{\epsilon_{i}\eta_{j}},w_{j},v_{i}]$$
(notice that in this product, the order of the factors is of no 
consequence).
\end{lem}
\begin{proof} Using the fact that any two conjugates of $t$ 
commute, we can write
\begin{eqnarray*}
	f(g(t)) & = & \prod_{i=1}^{m}v_{i}^{-1}\left(\prod_{j=1}^{n}
	w_{j}^{-1}t^{\eta_{j}}w_{j}\right)^{\epsilon_{i}}v_{i}\\
	{} & = & \prod_{i=1}^{m}\prod_{j=1}^{n}v_{i}^{-1}w_{j}^{-1}
	t^{\epsilon_{i}\eta_{j}}w_{j}v_{i}   \\
	{} & = & \prod_{i=1}^{m}\prod_{j=1}^{n}t^{\epsilon_{i}\eta_{j}}
	[t^{\epsilon_{i}\eta_{j}},w_{j}v_{i}].
\end{eqnarray*}
We conclude thanks to the relation 
$[x,yz]=[x,z][x,y][x,y,z]$.
\end{proof}
In a nilpotent group $G$ of class $\leq 2$, two conjugates of any 
element $t\in G$ commute. Therefore, as an immediate consequence of 
Lemma 2.1, we observe that any two polynomial automorphisms of $G$ 
commute. Since these automorphisms generate ${\rm P}(G)$, we obtain: 
%
\begin{cor} If $G$ is a nilpotent group of class $\leq 2$, then
${\rm P}(G)$ is abelian.
\end{cor}
We are now ready to prove our first theorem. 
\begin{proof}[Proof of Theorem 1.1.]  Since ${\rm P}(G)$ contains ${\rm I}(G)$ (which is 
nilpotent of class $k-1$ exactly), it suffices to show that ${\rm P}(G)$ is 
nilpotent of class at most $k-1$.
We argue by induction on the nilpotency class $k$ of $G$.
The case $k=2$ follows from Corollary 2.1. Now suppose that our theorem 
is proved for an integer $k\geq 2$ and consider a nilpotent group $G$ of class $k+1$.
Denote by $\zeta(G)$ the centre of $G$. One can define a homomorphism
$\Theta:{\rm P}(G)\to{\rm A}(G/\zeta(G))$, where for each 
$f\in{\rm P}(G)$, $\Theta(f)$ is the automorphism induced by $f$ in 
$G/\zeta(G)$. Clearly, if $f$ is a polynomial automorphism of $G$, then
$\Theta(f)$ is a polynomial automorphism of $G/\zeta(G)$. Hence
$\Theta({\rm P}(G))$ is a subgroup of ${\rm P}(G/\zeta(G))$, and so, 
by induction, is nilpotent of class at most $k-1$. Since 
$\Theta({\rm P}(G))$ and ${\rm P}(G)/\ker\Theta$ are isomorphic, it 
suffices to show that $\ker\Theta$ is included in the centre of ${\rm P}(G)$ 
and the theorem is proved. For that, consider an element 
$g\in\ker\Theta$ and put $w(x)=x^{-1}g(x)$ for any $x$ in $G$.
Thus $g(x)=xw(x)$ and $w(x)$ belongs to $\zeta(G)$ for all $x\in G$.
Notice that $w$ defines a homomorphism of $G$ into $\zeta(G)$ since
$$w(xy)=y^{-1}x^{-1}g(x)g(y)=y^{-1}w(x)g(y)=w(x)w(y).$$
In order to show that $g$ belongs to the centre of ${\rm P}(G)$, it 
suffices to verify that $g$ commutes with any polynomial automorphism $f$ of $G$.
Suppose that $f$ is defined by the relation
$$f(x)=(v_{1}^{-1}x^{\epsilon_{1}}v_{1})\ldots
(v_{m}^{-1}x^{\epsilon_{m}}v_{m}).$$
We have easily
$$f(g(x))=f(xw(x))=f(x)f(w(x))=f(x)w(x)^{\epsilon},$$
where $\epsilon=\epsilon_{1}+\cdots+\epsilon_{m}$. In the same way, 
by using the fact that $w$ is a homomorphism, we can write
\begin{eqnarray*}
	g(f(x)) & = & f(x)w(f(x))  \\
	{} & = & f(x)(w(v_{1})^{-1}w(x)^{\epsilon_{1}}w(v_{1}))\ldots
(w(v_{m})^{-1}w(x)^{\epsilon_{m}}w(v_{m})),
\end{eqnarray*}
whence $g(f(x))=f(x)w(x)^{\epsilon}$.
Thus $g$ and $f$ commute, as required, and the result follows.  
\end{proof}
Now we undertake the proof of our second theorem.
First we need the following result, which is 
well known and easy to prove 
(see for example \cite[Lemma 34.51]{NE} or \cite[ Part 2, p. 64]{RO}).
\begin{lem} In a metabelian group $G$, if $t$ is an element of the 
derived subgroup $[G,G]$, we have the relation $[t,x,y]=[t,y,x]$ for all 
$x,y\in G$. 
\end{lem}
We arrive to the key lemma in the proof of Theorem 1.2. This lemma 
shows that when $G$ is metabelian, any element 
$h\in [{\rm P}(G),{\rm P}(G)]$ operates trivially on $[G,G]$
and on $G/[G,G]$.
\begin{lem} Let $G$ be a metabelian group. Suppose that $h$ is an 
element of the derived subgroup $[{\rm P}(G),{\rm P}(G)]$. Then\\
{\rm (i)} $h(t)=t$ for all $t\in [G,G]$;\\ 
{\rm (ii)} $x^{-1}h(x)$ belongs to $[G,G]$ for all $x\in G$.
\end{lem}
\begin{proof} {\em (i)} Consider the homomorphism 
$\Phi:{\rm P}(G)\to {\rm A}([G,G])$ defined like this: for any 
$f\in{\rm P}(G)$, $\Phi(f)$ is the restriction of $f$ to $[G,G]$.
We must show that $\ker\Phi$ contains $[{\rm P}(G),{\rm P}(G)]$.
For that, first notice that any two conjugates of $t\in [G,G]$ 
commute since $G$ is metabelian. Now we apply Lemma 2.1.
If $f$ and $g$ are polynomial 
automorphisms of $G$ defined as in this lemma, we obtain the 
equalities
\begin{eqnarray*}
f(g(t))	& = & \prod_{i=1}^{m}\prod_{j=1}^{n}t^{\epsilon_{i}\eta_{j}}
[t^{\epsilon_{i}\eta_{j}},v_{i}]
[t^{\epsilon_{i}\eta_{j}},w_{j}]
[t^{\epsilon_{i}\eta_{j}},w_{j},v_{i}],  \\
g(f(t)) & = & \prod_{i=1}^{m}\prod_{j=1}^{n}t^{\epsilon_{i}\eta_{j}}
[t^{\epsilon_{i}\eta_{j}},v_{i}]
[t^{\epsilon_{i}\eta_{j}},w_{j}]
[t^{\epsilon_{i}\eta_{j}},v_{j},w_{i}],
\end{eqnarray*}
and so, by Lemma 2.2, $f(g(t))=g(f(t))$ for all  $t\in [G,G]$.
It follows that $[f,g]$ belongs to $\ker\Phi$. In other words, the images 
of $f$ and $g$ in ${\rm P}(G)/\ker\Phi$ commute. Since ${\rm P}(G)/\ker\Phi$ is 
generated by the images of the polynomial automorphisms, this quotient 
is abelian. It follows that $\ker\Phi$ contains $[{\rm P}(G),{\rm P}(G)]$, as 
desired.\\
{\em (ii)} Here, we consider the homomorphism 
$\Psi:{\rm P}(G)\to {\rm A}(G/[G,G])$, where for any $f\in {\rm P}(G)$, 
$\Psi(f)$ is the automorphism induced in $G/[G,G]$ by $f$.
Since a polynomial automorphism of $G$ induces in $G/[G,G]$ a polynomial automorphism of
$G/[G,G]$, $\Psi({\rm P}(G))$ is a subgroup of ${\rm P}(G/[G,G])$.
But ${\rm P}(G/[G,G])$ is abelian (see for instance Corollary 2.1 above) and  
$\Psi({\rm P}(G))$ is isomorphic to ${\rm P}(G)/\ker\Psi$. 
Hence ${\rm P}(G)/\ker\Psi$ is abelian. Consequently,
$\ker\Psi$ contains $[{\rm P}(G),{\rm P}(G)]$ 
and the result follows.
\end{proof}
%
\begin{proof}[Proof of Theorem 1.2.] Let $f,g$ be two elements of $[{\rm P}(G),{\rm P}(G)]$.
For any $x\in G$, put $v(x)=x^{-1}f(x)$ and $w(x)=x^{-1}g(x)$. 
By Lemma 2.3, $v(x)$ and  $w(x)$ belong to $[G,G]$.
Applying again Lemma 2.3, we can write
$$f(g(x)=f(xw(x))=f(x)f(w(x))=xv(x)w(x).$$
In the same way, we have $g(f(x))=xw(x)v(x)=xv(x)w(x)$.
It follows that $f$ and $g$ commute. Thus $[{\rm P}(G),{\rm P}(G)]$ 
is abelian, and so ${\rm P}(G)$ is metabelian.  
\end{proof} 
%
\section{IA-automorphisms of two-generator metabelian groups}
By way of illustration, we apply Theorem 1.2 to IA-automorphisms of a
two-generator metabelian group. We recall that an automorphism of a group $G$ is said to be an 
{\em IA-automorphism} if it induces the identity automorphism on $G/[G,G]$.
In a free metabelian group of rank 2, each IA-automorphism is inner
\cite{BA}, and so is a polynomial automorphism. It turns out that in 
any two-generator metabelian group, each IA-automorphism is polynomial.
This result is implicit in \cite{CA} with a different 
terminology. For convenience, we give a proof since this one is 
short and elementary.
%
\begin{prp} Each IA-automorphism
of a two-generator metabelian group is polynomial.
\end{prp}
To prove this proposition, we shall use the following result.
\begin{lem} In a metabelian group $G$, each  function $\varphi$ of the form
$$x\mapsto \varphi(x)=x[x,v_{1}]^{\eta_{1}}\ldots[x,v_{n}]^{\eta_{n}}\;\;(v_{i}\in G,\:
\eta_{i}\in{\mathbb Z})$$
is an endomorphism.
\end{lem} 
\begin{proof} Thanks to the relation $[xy,z]=y^{-1}[x,z]y[y,z]$, we get
$$ \varphi(xy)=xy\prod_{i=1}^n\left(y^{-1}[x,v_i]y[y,v_i]\right)^{\eta_i}.$$
But since the derived subgroup of $G$ is abelian, we can write
\begin{eqnarray*}
\varphi(xy) & = & xy\prod_{i=1}^{n}\left( y^{-1}[x,v_i]^{\eta_i}y\right) \prod_{i=1}^{n}[y,v_i]^{\eta_i} \\
{} & = & xy\left(  y^{-1} \left(\prod_{i=1}^{n} [x,v_i]^{\eta_i}\right) y\right) \prod_{i=1}^{n}[y,v_i]^{\eta_i}\\
{} & = & \varphi(x)\varphi(y),\\
\end{eqnarray*}
as required.
\end{proof}
\begin{proof}[Proof of Proposition 3.1.]  
Suppose that $G$ is a two-generator metabelian group
generated by $a$ and $b$.
If $f$ is an IA-automorphism of $G$, we have $f(a)=av$ and $f(b)=bw$, 
where $v$ and $w$ belong to the derived subgroup $[G,G]$.
Now notice that $[G,G]$ is the normal closure of $[a,b]$. 
Therefore, $[G,G]$ is generated by $[a,b]$ and the elements of 
the form $[a,b,u]$, with $u\in G$.  Hence $v$ 
and $w$ can be written in the form
\begin{eqnarray*}
	v & = & [a,b]^{\alpha}\prod_{i=1}^{n}
	[a,b,v_{i}]^{\lambda_{i}},   \\
	w & = & [a,b]^{\beta}\prod_{i=1}^{n}
	[a,b,w_{i}]^{\mu_{i}},
\end{eqnarray*}  
where 
$\alpha,\beta,\lambda_{1},\ldots,\lambda_{n},\mu_{1},\ldots,\mu_{n}$ are 
integers (possibly equal to 0).
By using the relation $[x,y,z]=[x,y]^{-1}[x,z]^{-1}[x,yz]$,
we obtain 
\begin{eqnarray*}
	v & = & [a,b]^{\alpha-\lambda}\prod_{i=1}^{n}
	[a,v_{i}]^{-\lambda_{i}}[a,bv_{i}]^{\lambda_{i}},   \\
	w & = & [a,b]^{\beta-\mu}\prod_{i=1}^{n}
	[a,w_{i}]^{-\mu_{i}}[a,bw_{i}]^{\mu_{i}},
\end{eqnarray*}  
where $\lambda=\lambda_{1}+\cdots+\lambda_{n}$ and $\mu=\mu_{1}+\cdots+\mu_{n}$.
Now put
$$\varphi(x)=x[x,b]^{\alpha -\lambda}[x,a]^{\mu-\beta}
\prod_{i=1}^{n}[x,v_{i}]^{-\lambda_{i}}[x,bv_{i}]^{\lambda_{i}}
[x,w_{i}]^{\mu_{i}}[x,aw_{i}]^{-\mu_{i}}.$$ 
By Lemma 3.1, $\varphi$ is an 
endomorphism of $G$. Moreover, we have
$$\varphi(a)=a[a,b]^{\alpha -\lambda}
\prod_{i=1}^{n}[a,v_{i}]^{-\lambda_{i}}[a,bv_{i}]^{\lambda_{i}}
=av=f(a)$$
since $[a,w_{i}]=[a,aw_{i}]$. 

In the same way, we get
$$\varphi(b)=b[a,b]^{\beta-\mu}
\prod_{i=1}^{n}[b,w_{i}]^{\mu_{i}}[b,aw_{i}]^{-\mu_{i}}.$$
By using the identity
$[a,w_{i}]^{-1}[a,bw_{i}]=[b,aw_{i}]^{-1}[b,w_{i}]$ (valid in any 
group), we obtain
$$\varphi(b)=b[a,b]^{\beta-\mu}
\prod_{i=1}^{n}[a,w_{i}]^{-\mu_{i}}[a,bw_{i}]^{\mu_{i}}=bw=f(b).$$
Thus $f=\varphi$ and the proof is complete. 
\end{proof}
We remark that Proposition 3.1 cannot be extended to 
three-generator metabelian groups.
For example, in the free metabelian group of rank 3 
freely generated by $a,b,c$, consider the 
IA-automorphism $f$ defined  by 
$f(a)=a$, $f(b)=b$ and $f(c)=c[a,b]$.
Suppose that $f$ is polynomial.
Since $[a,b]=c^{-1}f(c)$, the commutator $[a,b]$
would be in the normal closure of $c$, hence would be
a product of conjugates of $c^{\pm 1}$. Substituting 1 for $c$ in 
this expression gives then $[a,b]=1$, a contradiction. Therefore 
$f$ is an IA-automorphism which is not polynomial. 

As a consequence of Theorem 1.2 and Proposition 3.1, we obtain
an alternative proof of a result due to
C. K. Gupta \cite{GU} (see also \cite{CM}). 
%
\begin{cor}[\cite{GU}] In a two-generator metabelian group, 
the group of IA-automorphisms is metabelian.
\end{cor}
Let $M_{d}$ denote the free metabelian group of rank 
$d$. By a result of Bachmuth \cite{BA}, if $d\geq 3$, the group of IA-automorphisms 
of $M_{d}$ contains a subgroup which is (absolutely) free of rank $d$. Thus Corollary 
3.1 fails in a $d$-generator metabelian group when $d\geq 3$. Also Bachmuth's result
shows once again that the group of IA-automorphisms 
of $M_{d}$ is not included in ${\rm P}(M_{d})$ (if $d\geq 3$), since 
${\rm P}(M_{d})$ is metabelian.

In conclusion we mention that the metabelian groups constitute an 
important source of polynomial endomorphisms and automorphisms.
Indeed, by Lemma 3.1, each  function of the form
$$x\mapsto x[x,v_{1}]^{\eta_{1}}\ldots[x,v_{n}]^{\eta_{n}}\;\;(v_{i}\in G,\:
\eta_{i}\in{\mathbb Z})$$
is an endomorphism in a metabelian group $G$. Besides, when $G$ is metabelian 
and nilpotent, such an endomorphism is an automorphism since in a 
nilpotent group, every function of the form
$$x\mapsto u_{0}x^{\epsilon_{1}}u_{1}\ldots 
u_{m-1}x^{\epsilon_{m}}u_{m}\;\;(u_{i}\in G,\:
\epsilon_{i}\in{\mathbb Z})$$
is a bijection if $\epsilon_{1}+\cdots+\epsilon_{m}=\pm 1$
(see \cite[Theorem 1]{EN}).
 %
%

%

\begin{thebibliography}{10}
%
\bibitem{BA} S. Bachmuth,
Automorphisms of free metabelian groups,
{\em Trans. Amer. Math. Soc.} {\bf 118} (1965), 93--104.
%
\bibitem{CA} A. Caranti and C. M. Scoppola,
Endomorphisms of two-generated metabelian groups that induce the 
identity modulo the derived subgroup,
{\em Arch. Math.} {\bf 56} (1991), 218--227.
%
\bibitem{CM} F. Catino and M. M. Miccoli,
A note on IA-endomorphisms of two-generated metabelian groups,
{\em Rend. Sem. Mat. Univ. Padova} {\bf 96} (1996), 99--104.
%
\bibitem{CR} G. Corsi Tani and M. F. Rinaldi Bonafede, Polynomial 
automorphisms in nilpotent finite groups, {\em Boll. U.M.I.} {\bf 5} 
(1986), 285--292.
%
\bibitem{EN} G. Endimioni, Applications rationnelles d'un groupe nilpotent, 
{\em C. R. Acad. Sci. Paris} {\bf 314} (1992), 431--434.
%
\bibitem{GU} C. K. Gupta,
IA-automorphisms of two-generator metabelian groups,
{\em Arch. Math.} {\bf 37} (1981), 106--112.
%
\bibitem{KO} G. Kowol, Polynomautomorphismen von 
Gruppen, {\em Arch. Math.} {\bf 57} (1991), 114--121. 
%
\bibitem{NE} H. Neumann, {\em Varieties of Groups}, Springer-Verlag, Berlin 
(1967). 
%
\bibitem{RO} D. J. S. Robinson, {\em Finiteness Conditions and
Generalized Soluble Groups}, Springer-Verlag, Berlin (1972).
%
\bibitem{SC} D. Schweigert, Polynomautomorphismen auf endlichen 
Gruppen, {\em Arch. Math.} {\bf 29} (1977), 34--38.
%
\end{thebibliography}
\end{document}